\documentclass[a4paper,12pt]{article}
\usepackage{amsmath}
\usepackage{amsfonts}
\usepackage{supertabular}
\usepackage[latin9]{inputenc}
\usepackage[english]{varioref}
\usepackage{dcolumn}
\usepackage[height=22cm , width = 16cm , top = 4cm , left = 3cm, a4paper]{geometry}
\usepackage{amsthm}
\usepackage{dsfont}
\usepackage{hyperref}
\usepackage[hypcap]{caption}
\usepackage{color}

\numberwithin{equation}{section}

\theoremstyle{plain}
\newtheorem{thm}{Theorem}[section]

\newtheorem{lem}[thm]{Lemma}

\newtheorem{rmkk}[thm]{Remark}

\newcommand{\enter}{\bigskip}

\begin{document}
 \author{Prasanta Kumar Barik\footnote{{\it{${}$ Email address: }}pbarik@ma.iitr.ac.in/prasant.daonly01@gmail.com}\\
\footnotesize  Department of Mathematics,\\ \small {Indian Institute of Technology Roorkee,}\\ \small{ Roorkee-247667, Uttarakhand, India}
  }
\title {Existence of mass-conserving weak solutions to the singular coagulation equation with multiple fragmentation }

\maketitle

%%%%%%%%%%%%%%%%%%%%%%%%%%%% %%%%%%%%%%%%%%%%%
\hrule \vskip 8pt

\begin{quote}
{\small {\em\bf Abstract.} In this paper we study the continuous coagulation and multiple fragmentation equation for the mean-field description of a system of particles taking into account the combined effect of the coagulation and the fragmentation processes in which a system of particles growing by successive mergers to form a bigger one and a larger particle splits into a finite number of smaller pieces. We demonstrate the global existence of mass-conserving weak solutions for a wide class of coagulation rate, selection rate and breakage function. Here, both the breakage function and the coagulation rate may have algebraic singularity on both the coordinate axes. The proof of the existence result is based on a weak $L^1$ compactness method for two different suitable approximations to the original problem, i.e. the conservative and non-conservative approximations. Moreover, the mass-conservation property of solutions is established for both approximations.

   \enter
}

\end{quote}

\hspace{.5cm}

\vskip 10pt \hrule

%%%%%%%%%%%%%%%%%%%%%%%%%%%%%%%%%%%%%%%%%%%%%%%%%%%%%%%%%%%%%%%%%%%
%%%%%%%%%%%%%%%%%%%%%%%%%%%%%%%%%%%%%%%%%%%%%%%%%%%%%%%%%%%%%%%%%%%
\section{Introduction}\label{existintroduction1}
%%%%%%%%%%%%%%%%%%%%%%%%%%%%%%%%%%%%%%%%%%%%%%%%%%%%%%%%%%%%%%%%%%%
%%%%%%%%%%%%%%%%%%%%%%%%%%%%%%%%%%%%%%%%%%%%%%%%%%%%%%%%%%%%%%%%%%%
We investigate the existence of mass-conserving weak solutions to the continuous coagulation and multiple fragmentation equation. We first recall that the coagulation and multiple fragmentation equation (CFME) provides a mean-field description of a system of particles growing by successive mergers to form a larger one and a bigger particle splits into daughter particles. Each particle is fully identified by its volume (or size) $y \in \mathds{R}_{>0}$. Denoting by $g(t, y) \ge 0$, the concentration of particles of volume $y \in \mathds{R}_{>0}$ at time $t \ge 0$, the dynamics of $g$ is given by \cite{Melzak:1957, McLaughlin:1997, Giri:2011Thecontinuous, Giri:2012, Giri:2013, Camejo:2015, Laurencot:2018}
\begin{align}\label{Cmfe}
\frac{\partial g(t, y)}{\partial t}  = &\frac{1}{2} \int_{0}^{y} A(y-z, z)g(t, y-z)g(t, z)dz - \int_{0}^{\infty} A(y, z)g(t, y)g(t, z)dz\nonumber\\
  &+\int_y^{\infty}b(y|z)S(z)g(t, z)dz-S(y)g(t, y),
\end{align}
with the initial value
\begin{align}\label{Initialdata}
g(0, y) = g^{in}(y)\ge 0~ \mbox{a.e.}.
\end{align}
Here the non-negative and symmetric function $A(y, z)$ represents the coagulation rate which describes the rate at which the particles of volume $y$ unite with the particles of volume $z$ to produce the larger particles of volume $y+z$ whereas $b(y|z)$ is the breakage function which gives the contribution to the formation of particles of volume $y$ from the breakage of particles of volume $z$ and the selection rate $S(y)$ represents  the rate at which the particles of volume $y$ is selected to break. In addition, the breakage function is assumed to satisfy the following properties
 \begin{align}\label{TNP}
\int_0^zb(y|z)dy =N(z)\ \forall z \in \mathds{R}_{>0},\ \text{where}\ \underset{z\in \mathds{R}_{>0}}{\sup} N(z)=N <\infty \hspace{.2cm}  \text{and} \hspace{.2cm} b(y|z)=0 \hspace{.1cm} \forall \hspace{.1cm} y \ge z,
\end{align}
and
\begin{align}\label{MCP}
\int_0^zyb(y|z)dy =z,\ \ \ \forall y\in (0, z).
\end{align}
In (\ref{TNP}), $N(z)$ stands for the total number of daughter particles obtained from the breakage of particles of volume $z$ and is assumed its supremum is a finite constant $N \ge 2$. The condition (\ref{MCP}) ensures that the total volume (mass) in the system remains conserved during the fragmentation events.\\

The first term in \eqref{Cmfe} gives the production of particles of volume $y$ after coalescing of the particles of volumes $z$ and, $y-z$ due to the coagulation process whereas the second term shows the disappearance of the particles of volume $y$ after combining with the particles of volume $z$. The third and fourth terms describe the gain and loss of the particles of volume $y$ due to the multiple fragmentation events, respectively.\\

From \eqref{MCP}, it is clear that the total mass is conserved during the fragmentation process. Thus, we expect that the total mass will also conserve during both the coagulation and multiple fragmentation events. However, if the coagulation rate is very high compared to the fragmentation rate, the conservation of mass fails at a finite time due to the appearance of giant particles of the system. This process is called \emph{gelation transition} and the finite time at which this process occurs is known as the \emph{gelation time} \cite{Escobedo:2003, Leyvraz:1981}.\\

 Next, the total mass of the particles for coagulation and multiple fragmentation equation can be defined as
 \begin{align}\label{Totalmass}
\mathcal{M}_1(t)=\mathcal{M}_1(g(t)):=\int_0^{\infty}yg(t, y)dy,  \ \ t \ge 0.
\end{align}

 The well-posedness of weak solutions to the continuous CMFE with unbounded non-singular kernels have been investigated in many articles \cite{Dubovskii:1996, Giri:2011Thecontinuous, Giri:2012, McLaughlin:1997, Melzak:1957} and references therein. However, in \cite{Barik:2018Mass, Camejo:2015, Laurencot:2018, Norris:1999} the existence and uniqueness of solutions to the continuous CMFE with singular coagulation rates have been discussed. In particular, Norris \cite{Norris:1999} has studied the existence and uniqueness of solutions to the continuous SCE locally in time when coagulation kernel satisfy $A(y, z) \le \psi_1(y) \psi_1(z)$, with $\psi_1: (0, \infty) \rightarrow [0, \infty)$ and $\psi_1(ay)\le a\psi_1(y)$  for all $y \in (0, \infty)$, $a \ge 1$, where $\psi_1$ is a sub-linear function and the initial data $g^{in} \in L^1((0, \infty); \psi_1(y)^2)$. Moreover, the solutions satisfy the mass-conservation property for $\epsilon >0$ such that $\epsilon y \le \psi_1(y)$. Later in 2015, Camejo and Warnecke \cite{Camejo:2015} have discussed the existence of weak solutions to the continuous CMFE for the singular coagulation kernel, when the coagulation rate and selection rate, respectively, satisfy the following
 \begin{equation*}
 A_2(y, z) \le k (1+y)^{\lambda}(1+z)^{\lambda}({y}{z})^{-\sigma},\ \text{for}\  \sigma \in [0, 1/2), \lambda -\sigma \in [0,1)\ \text{and}\ k>0,
 \end{equation*}
 and
 \begin{equation*}
S_1(y) \le k' y^{\alpha}\ \text{where}\ \alpha \in (0, 1)\ \text{and}\ k'>0.
 \end{equation*}
 Moreover, they have shown the uniqueness result for this kernel $A_2$ when $\lambda =0$. Recently, Lauren\c{c}ot \cite{Laurencot:2018} has proven very interesting result to show the existence of mass-conserving solutions to the continuous CMFE by considering the breakage function, $b(y|z)=(\nu+2) \frac{y^{\nu}}{z^{1+\nu}}$, provided that $\nu \in (-2, -1]$. By taking on account of this breakage function, one can infer from \eqref{TNP} that an infinite number of particles are produced for $\nu=-1$ and on other hands, for $\nu \in (-2, -1)$, infeasible number of particles are created. Furthermore, a uniqueness result is established for restricted coagulation rate. Later, we have investigated the existence of mass-conserving solutions to the continuous SCE having linear growth for large volumes and singularity for small volume particles whatever the approximations to the original problems, see \cite{Barik:2018Mass}. In addition, we have relaxed the assumption on the initial data as in \cite{Norris:1999} to show the existence of solutions.\\

 Since the general uniqueness result to \eqref{Cmfe}--\eqref{Initialdata} is not available for singular coagulation rate $A$, breakage function, selection rate $S$ and initial data $g^{in}$ satisfying ($\Lambda_1$)--($\Lambda_4$) respectively,  it is not confirmed whether the solution to \eqref{Cmfe}--\eqref{Initialdata} obtained by a non-conservative approximation is mass conserving or not ? In \cite{Filbet:2004II}, Filbet and Lauren\c{c}ot have studied a finite volume scheme to discuss the gelation transition by using a non-conservative truncation. In addition, they have concluded that the loss of mass in the system decrease for a large domain. Hence, it is expected that when the upper limit of the truncated domain goes to infinity, then the mass conservation property holds for a non-conservative truncation. Later, in \cite{Filbet:2004}, they have established a mathematical proof of this numerical observation. A similar type of numerical observation for the coagulation-fragmentation equations (CFEs) by using a finite volume scheme has been discussed by Bourgade and Filbet in \cite{Bourgade:2008}. Recently, in \cite{Barik:2017Anote}, we have shown mathematically that a non-conservative coagulation and conservative fragmentation truncation for CFEs also gives the mass conserving solutions for certain classes of nonsingular unbounded coagulation and fragmentation kernels. The main novelty of the present work is to generalize the previous existing results in \cite{Barik:2017Anote, Barik:2018Mass, Camejo:2015}. In one hand, we have extended our previous work in \cite{Barik:2017Anote, Barik:2018Mass} to the continuous CMFE by using both conservative and non-conservative approximations. On the other hand, we have constructed a mass-conserving solution to the CMFE \eqref{Cmfe}--\eqref{Initialdata} which was an open problem in \cite{Camejo:2015} for the coagulation rate $A$ satisfies $(\Lambda_1)$ whatever the approximations. Moreover, we have also relaxed the assumption on the initial data as discussed in \cite{Norris:1999}. In addition, we have included $\alpha = 1$ in the selection rate $S_1$ in \cite{Camejo:2015}. The motivation of the present work is from  \cite{Barik:2017Anote, Barik:2018Mass, Camejo:2015, Laurencot:2018}. \\

Let us end the introductory section by describing the plan of the paper. In Section 2, we introduce some preliminary results, assumptions and statement of the main result i.e. Theorem \ref{TheoremCmfe}. In Section $3$, the existence and uniqueness of truncated solutions to \eqref{Cmfe} is shown by using both conservative and non-conservative truncation. In addition, the existence of mass-conserving weak solutions is proved by using a weak $L^1$ compactness technique in this section.

%%%%%%%%%%%%%%%%%%%%%%%%%%%%%%%%%%%%%%%%%%%%%%%%%%%%%%%%%%%%%%%%%%%%%%%%%%%%%%%%%%%%%%%%%%%%%%%%%%%%%%%%%%%%%%%%%%%%%%%%%%%%%%%%%%%%%%%%%%%%%%%%%%%%
%%%%%%%%%%%%%%%%%%%%%%%%%%%%%%%%%%%%%%%%%%%%%%%%%2%%%%%%%%%%%%%%%%%%%%%%%%%%%%%%%%%%%%%%%%%%%%%%%%%%%%%%%%%%%%%%%%%%%%%%%%%%%%%%%%%%%%%%%%%%%%%%%%%%
 \section{Assumptions, Preliminaries and Statement of the Main Result}
 Before stating the main result of this paper, we first describe the class of functions $g^{in}$, $A$, $S$ and $b$. More precisely, we assume that
the initial data $g^{in}$, $A$, $S$ and $b$ enjoy the following assumptions.\\

($\Lambda_1$)
 $A(y,z)\leq k_1\frac{(1+y+z) }{(yz)^{\beta}}$ for all $(y, z) \in \mathds{R}_{>0} \times \mathds{R}_{>0}$, $k_1\ge 0$ and $\beta \in [0, 1/2)$.\\
\\
($\Lambda_2$) there exists a positive constant $c_1>2$ (depending on $\nu$ and $\beta$) such that
\begin{eqnarray*}
\int_0^z y^{-2\beta}b(y|z)dy \leq c_1 z^{-2\beta},
\end{eqnarray*}
where $b(y|z)=(\nu+2)\frac{y^{\nu}}{z^{1+\nu}}$, for $-1< \nu \le 0$.\\

 Note: Throughout the paper we assume $b(y|z)=(\nu+2)\frac{y^{\nu}}{z^{1+\nu}}$, for $-1< \nu \le 0$.\\
\\
($\Lambda_3$) $S(y) \leq k_2 y^{1+\nu},\ \forall y\in \mathds{R}_{>0} $ for $k_2 \ge 0$ and there exists a $\gamma \in (1, 2)$ (depending on $\nu$ \& $\beta$) such that $ \gamma (\nu -\beta) +1>0 $, where $\nu$ is defined in $(\Lambda_2)$.\\

($\Lambda_4$) $g^{in} \in L^1_{-2\beta, 1}(\mathds{R}_{>0})$.\\

\begin{rmkk}
One can easily be checked that our coagulation rate is covering the Smoluchowski coagulation kernel in Brownian motion \cite{Aldous:1999}, formation of bubbles in stochastic stirred forths \cite{Clark:1999} and Granulation kernel \cite{Kapur:1972} in the existence result.
\end{rmkk}
Now, we are in a position to state the main theorem of this paper.
\begin{thm}\label{TheoremCmfe}
 Consider a function $g^{in}$ satisfying $(\Lambda_4)$ and assume that the functions $A$, $b$ and $S$ enjoy the assumptions  $(\Lambda_1)$--$(\Lambda_3)$.
Let $g_n$ be the solution to (\ref{tcfe}) for $n\ge 1$. Then there is a subsequence $(g_{n_k})$ of $(g_{n})$ and a mass conserving solution $g$ to \eqref{Cmfe}--\eqref{Initialdata} such that
 \begin{align}\label{TheoremEquation}
 g_{n_k}\to g \ \  \text{in}\ \mathcal{C}([0,T]_w; L_{-\beta, 1}^1(\mathds{R}_{>0} ) )\ \text{for each}\ T>0
 \end{align}
 satisfying the following weak formulation
 \begin{align}\label{definition}
\int_0^{\infty} [ g(t, y) - g^{in}(y)] \omega(y)dy=&\frac{1}{2}\int_0^t \int_0^{\infty} \int_{0}^{\infty}\tilde{\omega}(y, z) A(y, z)g(s, y) g(s, z)dzdyds\nonumber\\
&-\int_0^t \int_0^{\infty}\eta(y)S(y) g(s, y)dyds,
\end{align}

where
\begin{align}\label{Identity1}
\tilde{\omega} (y,z):=\omega(y+z)-\omega (y)-\omega(z)
\end{align}
and
\begin{align}\label{Identity2}
\eta(y):=\omega(y)-\int_0^y b(z|y)\omega(z)dz,
\end{align}
for every $t \in [0, T]$ and $\omega \in L^{\infty}(\mathds{R}_{>0})$.
 \end{thm}

Note: Here, the space of weakly continuous functions from $[0, T]$ to $L_{-\beta, 1}^1(\mathds{R}_{>0})$ is denoted by $\mathcal{C}( [0,T]_w;  L^1_{-\beta, 1}(\mathds{R}_{>0}) )$, and that a sequence $(g_n)$ converges to $g$ in $\mathcal{C}( [0,T]_w;  L^1_{-\beta, 1}(\mathds{R}_{>0}) )$ if
\begin{align}
\lim_{n \to \infty} \sup_{t \in [0,T]}\bigg|\int_0^{\infty}(y^{-\beta}+y)[g_n(t, y)-g(t, y)] \omega(y)dy\bigg|=0,
\end{align}
for every $\omega \in L^{\infty}(\mathds{R}_{>0})$.\\

 In order to prove Theorem \ref{TheoremCmfe} we need to define a particular class of convex functions denoted as $\mathcal{C}_{VP, \infty}$. Let us consider non-negative and convex functions $\sigma_1, \sigma_2  \in \mathcal{C}^{\infty}([0, \infty))$ and belong to the class $\mathcal{C}_{VP, \infty}$, if they enjoy the following properties:
\begin{description}
  \item[(i)] $\sigma_j(0)=\sigma_j'(0)=0$ and $\sigma_j'$ is concave;
  \item[(ii)] $\lim_{p \to \infty} \sigma_j'(p) =\lim_{p \to \infty} \frac{ \sigma_j(p)}{p}=\infty$;
  \item[(iii)] for $\gamma \in (1, 2)$,
  \begin{align*}
  S_{\gamma}(\sigma_j):= \sup_{p \ge 0} \bigg\{   \frac{ \sigma_j(p)}{p^{\gamma}} \bigg\} < \infty,
  \end{align*}
  for $j=1, 2$.
\end{description}
%Example of $\mathcal{C}_{VP, \infty}$ is $(p+1) \ln(p+1)-p$.\\

Further, since $g^{in}\in L_{-2\beta, 1}^1(\mathds{R}_{>0})$, then a refined version of de la Vall\'{e}e-Poussin theorem see  \cite[Theorem~2.8]{Laurencot:2015} ensures that there exist two non-negative functions $\sigma_1$ and $\sigma_2$ in $\mathcal{C}_{VP, \infty}$ with
\begin{align}\label{convexp1}
\sigma_i(0)=0,~~~\lim_{p \to {\infty}}\frac{\sigma_i(p)}{p}=\infty,~~~~i=1,2
\end{align}
and
\begin{align}\label{convexp2}
\Gamma_1 := \int_0^{\infty}\sigma_1(y)g^{in}(y)dy<\infty,~~\text{and}~~\Gamma_2 :=\int_0^{\infty}{\sigma_2(y^{-\beta}g^{in}(y))}dy<\infty.
\end{align}

Let us recall some additional properties of $\mathcal{C}_{VP, \infty}$ which are also required to prove Theorem \ref{TheoremCmfe}.
\begin{lem} Consider $\sigma_1$, $\sigma_2$ in $\mathcal{C}_{VP, \infty}$. Then we have the following results
\begin{equation}\label{convexp3}
\hspace{-5cm} \sigma_2(p_1)\le p_1\sigma'_2(p_1)\le 2\sigma_2(p_1),
\end{equation}
\begin{equation}\label{convexp4}
\hspace{-5.5cm} p_1\sigma_2'(p_2)\le \sigma_2(p_1)+\sigma_2(p_2),
\end{equation}
and
\begin{equation}\label{convexp5}
0 \le \sigma_1(p_1+p_2)-\sigma_1(p_1)-\sigma_1(p_2)\le  2\frac{p_1\sigma_1(p_2)+p_2\sigma_1(p_1)}{p_1+p_2},
\end{equation}
 for all $p_1, p_2 \in \mathds{R}_{>0}$.
\end{lem}
\begin{proof}
This lemma can be easily proved in a similar way as given in \cite{Barik:2017Anote, Barik:2018Mass, Filbet:2004, Laurencot:2018}.
\end{proof}

%%%%%%%%%%%%%%%%%%%%%%%%%%%%%%%%%%%%%%
\section{Existence of weak solutions}
%%%%%%%%%%%%%%%%%%%%%%%%%%%%%%%%%%%%%%

 In this section, we  construct a mass conserving solution relies on both the conservative and non-conservative approximations to \eqref{Cmfe}--\eqref{Initialdata} which is defined as: for a given natural number $n\in\mathds{N}$, we set
\begin{align}\label{Initialtrunc}
g_n^{in}(y)=g^{in}(y)\chi_{(0, n)}(y),
\end{align}
for  $\zeta \in \{0, 1\}$,
\begin{equation}\label{CoagKerTrun}
A_n^\zeta(y, z) := A(y, z) \chi_{(1/n, n)}(y) \chi_{(1/n, n)}(z) \left[ 1 - \zeta + \zeta \chi_{(0, n)}(y+z) \right],
\end{equation}
and
\begin{align}\label{SelectionTrunc}
S_n^c(y)=S(y)\chi_{(0, n )}(y).
\end{align}

Using \eqref{Initialtrunc}, \eqref{CoagKerTrun} and \eqref{SelectionTrunc}, we can rewrite \eqref{Cmfe}--\eqref{Initialdata} as
\begin{align}\label{tcfe}
\frac{\partial g_n(t, y)}{\partial t}  = &\frac{1}{2} \int_{0}^{y} A_n^\zeta(y-z, z) {g_n}(t, y-z) {g_n}(t, z)dz - \int_{0}^{n-\zeta y} A_n^\zeta(y, z) g_n(t, y) g_n(t, z)dz\nonumber\\
  &+\int_y^{n}b(y|z)S_n^c(z) g_n(t, z)dz-S_n^c(y) g_n(t, y),
\end{align}
with the truncated initial condition
\begin{align}\label{1tnin1}
g_n(0, y)=g_n^{in},  ~~ \text{for}~~ y\in (0, n).
\end{align}

For $\zeta=1$, the existence and uniqueness of a nonnegative solution $g_n \in \mathcal{C}'( [0, \infty); L^1(0, n) )$ to (\ref{tcfe})--(\ref{1tnin1}) can easily obtained by a classical fixed point theorem, see \cite{Camejo:2015, Giri:2011Thecontinuous, Giri:2012, Stewart:1989}. Moreover, $g_n$ enjoys a truncated version of mass conservation property, i.e.
\begin{align}\label{tmc}
\int_0^n y g_n(t, y)dy= \int_0^n yg_n^{in}(y)dy ~~~~ \text{for~all} ~~t> 0,
\end{align}
and for $\zeta=0$, we may follow \cite{Barik:2018Mass, Camejo:2015} to show the existence and uniqueness of a nonnegative solution $g_n \in \mathcal{C}'( [0, \infty); L^1(0, n) )$ to (\ref{tcfe})--(\ref{1tnin1}) and it satisfies
\begin{align}\label{ntmc}
\int_0^n y g_n & (t, y)dy=   \int_0^n yg_n^{in}(y)dy\nonumber\\
& -\frac{1}{2}\int_0^t \int_0^n \int_{n-y}^n (y+z)  A(y, z) \chi_{(1/n, n)}(y) \chi_{(1/n, n)}(z)  g_n(s, y)g_n(s, z)dz dy ds.
\end{align}

We next recall that, for $n\geq 1$, and $\omega \in L^{\infty}(\mathbb{R}_{>0})$, the solution $g_n$ to (\ref{tcfe})--(\ref{1tnin1}) satisfies the following weak formulation
\begin{align}\label{nctp3}
\int_0^n[g_n& (t, y)-g_n^{in}(y)]\omega(y)dy= -\int_0^t\int_0^n H_{\omega}(z)S_n(z)g_n(s, z)dzds \nonumber\\
& +\frac{1}{2} \int_0^t \int_0^n\int_0^{n} G_{\omega, n}^{\zeta}(y, z) \chi_{(1/n, n)}(y) \chi_{(1/n, n)}(z)  A(y, z) g_n(s, y)g_n(s, z)dzdyds,
\end{align}
where
\begin{align}\label{G Omega}
 G_{\omega, n}^{\zeta}(y, z)= \omega (y+z)\chi_{(0, n)}(y+z)-[\omega(y)+\omega(z)](1-\zeta +\zeta \chi_{(0, n)}(y+z))
 \end{align}
and
\begin{align}\label{H Omega}
H_{\omega}(z)=\omega(z)-\int_0^z \omega(y)b(y|z)dy.
\end{align}

Next our aim to show that the family of solutions $\{ g_n\}_{n \ge 1}$ is relatively compact in $\mathcal{C}([0,T]_w; L_{-\beta, 1}^1(\mathds{R}_{>0} ) )$. For that purpose, we apply the weak $L^1$ compactness method which is used in the pioneering work of Stewart \cite{Stewart:1989}. In the next lemma, we show the family of solutions $\{ g_n\}_{n \ge 1}$ is uniform bounded in $L_{-2\beta, 1}^1(\mathds{R}_{>0})$.

\subsection{Uniform Bound}
\begin{lem}\label{G(T)}
 Assume $(\Lambda_1)$--$(\Lambda_4)$ hold. Let $T>0$, then there is a constant $ \mathcal{G}(T)$ depending on $T$ such that
\begin{align*}
\int_0^{\infty} (y^{-2\beta}+y)g_n(t, y)dy \leq \mathcal{G}(T)\ \ \text{for all}\ t \in[0, T].
\end{align*}
\end{lem}
\begin{proof}
We take $\omega(y)=(y^{-2\beta}+y) \chi_{(0, n)}(y)$, and inserting it into (\ref{nctp3}) to obtain
\begin{align}\label{bound1}
\int_0^n (y^{-2\beta}+y) & [g_n(t, y)-g_n^{in}(y)]dy =- \int_0^t\int_0^n H_{\omega}(z)S_n(z)g_n(s, z)dz ds\nonumber\\
+&\frac{1}{2} \int_0^t \int_0^n\int_0^n G_{\omega, n}^{\zeta}(y, z) \chi_{(1/n, n)}(y) \chi_{(1/n, n)}(z)  A(y, z) g_n(s, y)g_n(s, z)dz dy ds.
\end{align}
We simplify $G_{\omega, n}^{\zeta}$ and $H_{\omega}$ separately. Next, on the first case for $y+z< n$ follows from \eqref{G Omega} that
\begin{align*}
G_{\omega, n}^{\zeta}(y, z)= & \omega (y+z)\chi_{(0, n)}(y+z)-[\omega(y)+\omega(z)](1-\zeta +\zeta \chi_{(0, n)}(y+z))\nonumber\\
\le &(y+z)^{-2\beta}+y+z-(y^{-2\beta}+y)-(z^{-2\beta}+z)\nonumber\\
\le &(y^{-2\beta}+z^{-2\beta}+y+z)-(y^{-2\beta}+y)-(z^{-2\beta}+z) =0.
\end{align*}
On the other case for $y+z \ge n$, \eqref{G Omega} yield
\begin{align*}
G_{\omega, n}^{\zeta}(y, z)= & \omega (y+z)\chi_{(0, n)}(y+z)-[\omega(y)+\omega(z)](1-\zeta +\zeta \chi_{(0, n)}(y+z))\nonumber\\
=&0-[y^{-2\beta}+y+z^{-2\beta}+z](1-\zeta) \le 0.
\end{align*}
We estimate $H_{\omega}(z)$, by using \eqref{H Omega}, $(\Lambda_2)$ and \eqref{MCP}, as
 \begin{align}\label{bound2}
 H_{\omega}(z)& =(z^{-2\beta}+z)-\int_0^z(y^{-2\beta}+y)b(y|z)dy\nonumber\\
 & \ge  z^{-2\beta}+z -c_1z^{-2\beta}-z= (1-c_1) z^{-2\beta}.
 \end{align}
Since $G_{\omega, n}^{\zeta}$ is non-positive for both above cases, thus one can infer that the second integral on the right-hand side of (\ref{bound1}) is non-positive. Then using (\ref{bound2}), $(\Lambda_3)$ and $(\Lambda_4)$ into (\ref{bound1}), we evaluate
\begin{align*}
\int_0^n (y^{-2\beta}+y) g_n(t, y)dy \le & \int_0^n (y^{-2\beta}+y) g_n^{in}(y)dy + k_2 (c_1-1) \int_0^t\int_0^n z^{1+\nu-2\beta} g_n(s, z)dzds\\
\le & \int_0^{\infty} (y^{-2\beta}+y) g^{in}(y)dy + k_2 (c_1-1) \int_0^t\int_0^n ( z^{-2\beta}+z) g_n(s, z)dzds \\
\le & \|g_0 \|_{L^1_{-2\beta, 1}(\mathds{R}_{>0})  }+ k_2 (c_1-1) \int_0^t\int_0^n ( z^{-2\beta}+z) g_n(s, z)dzds.
\end{align*}
Finally, an application of Gronwall's inequality gives
\begin{align*}
\int_0^n (y^{-2\beta}+y) g_n(t, y)dy \le \mathcal{G}(T),
\end{align*}
where $\mathcal{G}(T):=\|g_0 \|_{L^1_{-2\beta, 1}(\mathds{R}_{>0})  } e^{k_2 (c_1-1)T}$, for each $n \in \mathds{N}$. This proves Lemma \ref{G(T)}.
\end{proof}

In the coming lemma, we discuss the behaviour of $g_n$ for large volume particle $y$.
%%%%%%%%%%%%%%%%%%%%%%%%%%%%%%%%%%%%%%%%%%%%%%%%%%%%%%%%%%%
%%%%%%%%%%%%%%%%%%%%%%%%%%%%%%%%%%%%%%%%%%%%%%%%%%%%%%%%%
\begin{lem}\label{MassConCoagMulti}
Assume that the coagulation rate, breakage function, selection rate and initial data satisfy $(\Lambda_1)$--$(\Lambda_4)$, respectively. Then for every $n\ge 1$ and for $T>0$,
\begin{equation}\label{C(T)}
\sup_{t\in [0, T]}\int_0^n \sigma_1 (y)g_n(t, y)dy  \le \Theta(T),
\end{equation}

\begin{equation}\label{C(T1)}
(1-\zeta)\int_0^T \int_0^n\int_{n-y}^n \sigma_1 (y) \chi_{(1/n, n)}(y) \chi_{(1/n, n)}(z) A(y, z)g_n(s, y)g_n(s, z) dz dy ds\leq \Theta(T),
\end{equation}
and
\begin{equation}\label{C(T2)}
\sup_{t\in [0, T]} \int_0^t \int_0^n \frac{z\sigma_1'(y) -\sigma_1(y)}{\nu+3} S_n(y) g_n(s, y) dy ds \le \Theta(T),
\end{equation}
where $\Theta(T)$ (depending on $T$) is a positive constant and the $\sigma_1 \in \mathcal{C}_{VP, \infty}$ satisfies \eqref{convexp1} and \eqref{convexp2}.
\end{lem}

\begin{proof} We set $\omega (y)=\sigma_1(y) \chi_{(0, n)}(y)$, and inserting it into (\ref{nctp3}) to obtain
\begin{align}\label{large1}
\int_0^n \sigma_1(y)& g_n(t, y)dy =\int_0^n \sigma_1(y)g_n^{in}(y)dy -\int_0^t\int_0^n H_{\sigma_1}(z)S_n(z)g_n(s, z)dzds\nonumber\\
& +\frac{1}{2} \int_0^t \int_0^n\int_0^{n} G_{\sigma_1, n}^{\zeta}(y, z) \chi_{(1/n, n)}(y) \chi_{(1/n, n)}(z)  A(y, z) g_n(s, y)g_n(s, z)dzdyds,
\end{align}
where
\begin{align}\label{large2}
 G_{\sigma_1, n}^{\zeta}(y, z)= \sigma_1 (y+z)\chi_{(0, n)}(y+z)-[\sigma_1 (y)+\sigma_1 (z)](1-\zeta +\zeta \chi_{(0, n)}(y+z))
 \end{align}
and
\begin{align*}
H_{\sigma_1}(z)= \sigma_1(z)- \int_0^z \sigma_1(y)b(y|z)dy.
\end{align*}
By using \eqref{convexp2} and \eqref{Initialtrunc} into \eqref{large1}, we have
\begin{align}\label{large3}
\int_0^n \sigma_1(y)g_n(t, y)dy \le &\Gamma_1+\frac{1}{2}\int_0^t [P_n(s)+Q_n(s)]ds\nonumber\\
&- \int_0^t \int_0^n H_{\sigma_1}(z)  S_n(z)g_n(s, z)dzds,
\end{align}
where
\begin{align*}
P_n(s)= \int_0^n \int_0^{n-y}G_{\sigma_1, n}^{\zeta}(y, z) \chi_{(1/n, n)}(y) \chi_{(1/n, n)}(z)  A(y, z) g_n(s, y)g_n(s, z)dzdy,
\end{align*}
and
\begin{align*}
Q_n(s)=(1-\zeta) \int_0^n \int_{n-y}^n G_{\sigma_1, n}^{\zeta}(y, z) \chi_{(1/n, n)}(y) \chi_{(1/n, n)}(z)  A(y, z) g_n(s, y)g_n(s, z)dzdy.
\end{align*}
Multiplying $A(y, z)$ with \eqref{large2}, we have
\begin{align}\label{large4}
A(y, z)  G_{\sigma_1, n}^{\zeta}(y, z)= & A(y, z)  [\sigma_1 (y+z)\chi_{(0, n)}(y+z)\nonumber\\
     &-[\sigma_1 (y)+\sigma_1 (z)](1-\zeta +\zeta \chi_{(0, n)}(y+z))].
%\nonumber\\
% \le & A_n(y,z)[ \sigma_1 (y+z)-\sigma_1(y)-\sigma_1(z)]\nonumber\\
% \le &2k_1 \frac{(1+y+z)}{(yz)^{2\beta}} \times  \frac{y\sigma_1(z)+z\sigma_1(y)}{y+z}.
\end{align}

Next, we estimate $P_n(s)$, by using \eqref{large4}, \eqref{convexp5}, and ($\Lambda_1$), as
\begin{align}\label{Pn}
P_n(s)=&   \int_0^n \int_0^{n-y}     A(y, z)  \chi_{(1/n, n)}(y) \chi_{(1/n, n)}(z) [\sigma_1 (y+z)-\sigma_1 (y)-\sigma_1 (z)] \nonumber\\
  &~~~~~~~~~~~~~~\times g_n(s, y)g_n(s, z)dzdy \nonumber\\
\le &2k_1\int_0^n \int_0^{n-y} \frac{(1+y+z)}{(yz)^{\beta}} \times  \frac{y\sigma_1(z)+z\sigma_1(y)}{y+z}  g_n(s, y)g_n(s, z)dzdy \nonumber\\
%\le & 4k_1 \int_0^1 \int_0^{1}\frac{(1+y+z)}{(yz)^{\beta}} \times  \frac{y\sigma_1(z)+z\sigma_1(y)}{y+z} g_n(s, y)g_n(s, z)dzdy \nonumber\\
%& + 8k_1 \int_1^n \int_0^{1}\frac{(1+y+z)}{(yz)^{\beta}} \times  \frac{y\sigma_1(z)+z\sigma_1(y)}{y+z}  g_n(s, y)g_n(s, z)dzdy \nonumber\\
%& + 4k_1 \int_1^n \int_1^{n}\frac{(1+y+z)}{(yz)^{\beta}} \times  \frac{y\sigma_1(z)+z\sigma_1(y)}{y+z}  g_n(s, y)g_n(s, z)dzdy \nonumber\\
\le & 24 k_1 \int_0^1 \int_0^{1} (yz)^{-\beta}   \frac{y\sigma_1(z) }{y+z} g_n(s, y)g_n(s, z)dzdy  \nonumber\\
& + 24 k_1 \int_1^n \int_0^{1}\frac{z}{(yz)^{\beta}} \times  \frac{y\sigma_1(z)+z\sigma_1(y)}{y+z} g_n(s, y)g_n(s, z)dzdy  \nonumber\\
& + 8k_1 \int_1^n \int_1^{n}  \frac{y\sigma_1(z)+z\sigma_1(y)}{(yz)^{\beta}} g_n(s, y)g_n(s, z)dzdy.
\end{align}
Let us estimate the first term on the right-hand side of \eqref{Pn}, by using Lemma \ref{G(T)} and monotonicity of $\sigma_1$, as
\begin{align}\label{Pn1}
 24 k_1 \int_0^1 \int_0^{1}  (yz)^{-\beta}   \frac{y\sigma_1(z) }{y+z} g_n(s, y)g_n(s, z)dzdy
 %\le & 24 k_1 \sigma_1(1)  \int_0^1 \int_0^{1}  (yz)^{-2\beta}   g_n(s, y)g_n(s, z)dzdy
 \le 24 k_1 \sigma_1(1)\mathcal{G}(T)^2.
\end{align}
Again by using Lemma \ref{G(T)} and monotonicity of $\sigma_1$, the second term on the right-hand side of \eqref{Pn} can be evaluated as
\begin{align}\label{Pn2}
 24 k_1 \int_1^n \int_0^{1} & \frac{z}{(yz)^{\beta}} \times  \frac{y\sigma_1(z)+z\sigma_1(y)}{y+z} g_n(s, y)g_n(s, z)dzdy \nonumber\\
% = & 24 k_1 \int_1^n \int_0^{1}\frac{z}{(yz)^{\beta}} \times  \frac{y\sigma_1(z)}{y+z} g_n(s, y)g_n(s, z)dzdy  \nonumber\\
%&+24 k_1 \int_1^n \int_0^{1}\frac{z}{(yz)^{\beta}} \times  \frac{z\sigma_1(y)}{y+z} g_n(s, y)g_n(s, z)dzdy  \nonumber\\
\le & 24 k_1 \sigma_1(1) \int_1^n \int_0^{1} z^{-2\beta}  g_n(s, y)g_n(s, z)dzdy  \nonumber\\
&+24 k_1 \int_1^n \int_0^{1}  z^{-2\beta}  \sigma_1(y) g_n(s, y)g_n(s, z)dzdy  \nonumber\\
\le & 24 k_1 \sigma_1(1) \mathcal{G}(T)^2 +24 k_1 \mathcal{G}(T) \int_0^n   \sigma_1(y) g_n(s, y) dy.
\end{align}
Finally, we evaluate the last integral on the right-hand to \eqref{Pn}, by applying Lemma \ref{G(T)}, as
\begin{align}\label{Pn3}
 8k_1 \int_1^n \int_1^{n}  \frac{y\sigma_1(z)+z\sigma_1(y)}{(yz)^{\beta}} g_n(s, y)g_n(s, z)dzdy
%= & 16 k_1 \int_1^n \int_1^{n}   \frac{y\sigma_1(z) }{(yz)^{\beta}} g_n(s, y)g_n(s, z)dzdy \nonumber\\
\le & 16 k_1 \int_1^n \int_1^{n}  y\sigma_1(z) g_n(s, y)g_n(s, z)dzdy\nonumber\\
\le & 16 k_1 \mathcal{G}(T) \int_0^n   \sigma_1(y) g_n(s, y) dy.
\end{align}
Inserting \eqref{Pn1}, \eqref{Pn2} and \eqref{Pn3} into \eqref{Pn}, we obtain
\begin{align}\label{Pnf}
P_n(s) \le & 48 k_1 \sigma_1(1)\mathcal{G}(T)^2 + 40 k_1 \mathcal{G}(T) \int_0^n   \sigma_1(y) g_n(s, y) dy.
\end{align}
If $y+z \geq n$, then
\begin{align}\label{bountct}
G_{\sigma_1}(y,z)=-\sigma_1(y)-\sigma_1(z).
\end{align}
Using (\ref{bountct}), $Q_n(s)$ can be rewritten as
\begin{align}\label{Qn}
Q_n(s)=& - 2 k_1(1-\zeta) \int_0^n\int_{n-y}^{n} \sigma_1(y)  \chi_{(1/n, n)}(y) \chi_{(1/n, n)}(z) A(y, z)g_n(s, y)g_n(s, z)dz dy\nonumber\\
 \le & 0.
\end{align}
Since $\sigma_1$ is a non-decreasing convex function and its derivative is concave, then we estimate the $H_{\sigma_1}(z)$, by using \eqref{MCP}, as
\begin{align}\label{Hn}
H_{\sigma_1}(z)=&\sigma_1(z)-\int_0^z \sigma_1(y)b(y|z)dy\nonumber\\
= &  \int_0^z \bigg[ \frac{\sigma_1(z)}{z} y b(y|z) - \frac{\sigma_1(y)}{y} y b(y|z) \bigg]dy \nonumber\\
= &  \int_0^z \bigg[ \frac{\sigma_1(z)}{z}  - \frac{\sigma_1(y)}{y} \bigg] yb(y|z) dy \ge   \int_0^z \bigg( \frac{\sigma_1(z)}{z} \bigg)' (z-y) y b(y|z) dy \nonumber\\
= &  \frac{z\sigma_1'(z) -\sigma_1(z) } {z^2} \int_0^z  (z-y)y b(y|z) dy = \frac{z\sigma_1'(z) -\sigma_1(z)}{\nu+3}.
\end{align}
Inserting (\ref{Pnf}), (\ref{Qn}), and (\ref{Hn}) into (\ref{large3}), we obtain
\begin{align*}
\int_0^n \sigma_1(y)g_n(t, y)dy  & + \int_0^t \int_0^n \frac{y\sigma_1'(y) -\sigma_1(y)}{\nu+3} S_n(y) g_n(s, y) dy ds \nonumber\\
 &+ 2k_1(1-\zeta) \int_0^t \int_0^n\int_{n-y}^{n} \sigma_1(y) \chi_{(1/n, n)}(y) \chi_{(1/n, n)}(z)\nonumber\\
  &~~~~~~~~~~~~~~~~~~~\times A(y, z)g_n(s, y)g_n(s, z)dzdy ds \nonumber\\
  \le & \Gamma_1+ 48 k_1 \sigma_1(1)\mathcal{G}(T)^2 T+ 40 k_1 \mathcal{G}(T) \int_0^t \int_0^n   \sigma_1(y) g_n(s, y) dy ds.
\end{align*}
Then by Gronwall's inequality, we get
\begin{align*}
\int_0^n \sigma_1(y) & g_n(t, y)dy   + \int_0^t \int_0^n \frac{y\sigma_1'(y) -\sigma_1(y)}{\nu+3} S_n(y) g_n(s, y) dy ds \nonumber\\
 &+ 2k_1(1-\zeta) \int_0^t \int_0^n\int_{n-y}^{n} \sigma_1(y) \chi_{(1/n, n)}(y) \chi_{(1/n, n)}(z) A(y, z)g_n(s, y)g_n(s, z)dzdy ds \nonumber\\
 \le & \Theta(T),
\end{align*}
where $\Theta(T)=  (\Gamma_1+48 k_1 \sigma_1(1)\mathcal{G}(T)^2 T ) e^{40 k_1 \mathcal{G}(T) T} $,
which completes the proof of Lemma \ref{MassConCoagMulti}.

\end{proof}

%%%%%%%%%%%%%%%%%%%%%%%%%%%%%%%%%%%%%%%%%%%%%%%%%%
%%%%%%%%%%%%%%%%%%%%%%%%%%%%%%%%%%%%%%%%%%%%%%%%%%

In order to apply a refined version of de la Vall\`{e}e Poussin theorem \cite{Laurencot:2015} to show the equi-integrability condition for the family of solutions $\{ g_n\}_{n > 1}$, we require the following lemma.
\subsection{Equi-integrability}
\begin{lem}\label{last}
 Assume $(\Lambda_1)$--$(\Lambda_4)$ hold. Let $T>0$ and $n\ge R>1$, there is a constant $C(T, R)$ such that
\begin{align*}
(i)\ \ \ \ \sup_{t\in [0,T]}\int_0^R \sigma_2( y^{-\beta}  g_n(t, y))dy \leq C(T, R),
\end{align*}
%\begin{align*}
%\sup_{t\in [0,T]}\bigg| \frac{d}{dt}\int_0^R y^{-\beta}g_n(t, y) \omega(y)dy \bigg| \le  \|\omega\|_{L^{\infty}(0,R)} C(R,T),
%\end{align*}
where $\omega \in L^{\infty}(0,R)$ and the $\sigma_2 \in \mathcal{C}_{VP, \infty}$ satisfies \eqref{convexp1} and \eqref{convexp2}.\\

For every $\epsilon >0$ depending on $R_{\epsilon}>1$ such that
\begin{align*}
\hspace{-2.5cm}(ii)\ \ \ \ \sup_{t\in [0,T]}\int_{R_{\epsilon}}^{\infty}   g_n(t, y)dy \le \epsilon.
\end{align*}
\end{lem}

\begin{proof} We take $h_n(t, y):= y^{-\beta}  g_n(t, y)$ and $n \ge R$. Next, by using the Leibniz's integral rule and \eqref{tcfe}, we have
%\begin{align}\label{Equintp1}
%\frac{d}{dt}\int_0^R \sigma_2(h_n(t, y))dy %=&\frac{d}{dt}\int_0^R \sigma_2( y^{-\beta}  g_n(t, y))dy =  \int_0^R \sigma_2'(h_n(t, y)) y^{-\beta}   \frac{\partial}{\partial t} g_n(t, y)dy\nonumber\\
%=& \int_0^R \sigma_2'(h_n(t, y))  y^{-\beta} \bigg[ \frac{1}{2} \int_0^y A_n^{\zeta}(y-z, z)g_n(t, y-z)g_n(t, z)dz\nonumber\\
%&-\int_0^n A_n^{\zeta}(y, z)g_n(y, t)g_n(z, t)dz+\int_y^n b(y|z)S_n^c(z)g_n(t, z)dz\nonumber\\
%&-S_n^c(y)g_n(t, y)\bigg]dy.
%\end{align}
%
%Since the second and the fourth terms to the right hand side of \eqref{Equintp1} are non-positive, then we obtain
\begin{align}\label{Equintp2}
\frac{d}{dt}\int_0^R \sigma_2(h_n(t, y))dy \le &\frac{1}{2} \int_0^R \int_0^y \sigma_2'(h_n(t, y))  y^{-\beta} A_n^{\zeta}(y-z, z)g_n(t, y-z)g_n(t, z)dzdy\nonumber\\
&+\int_0^R \int_y^n \sigma_2'(h_n(t, y))  y^{-\beta}  b(y|z) S_n(z) g_n(t, z)dzdy.
\end{align}
Changing the order of integration by using Fubini's theorem and simplifying it further, by substituting $y-z=y'$ and $z=z'$, as
\begin{align}\label{Equintp3}
\frac{d}{dt}\int_0^R \sigma_2(h_n(t, y))dy \leq &\frac{1}{2} \int_0^R \int_0^{R-z}\sigma_2'(h_n(t, y+z)) (y+z)^{-\beta} A_n^{\zeta}(y, z)g_n(t, y)g_n(t, z)dydz\nonumber\\
&+\int_0^R \int_0^z  y^{-\beta} b(y|z)S_n(z) \sigma_2'(h_n(t, y))g_n(t, z)dydz\nonumber\\
&+\int_R^n \int_0^R  y^{-\beta} b(y|z)S_n(z) \sigma_2'(h_n(t, y)) g_n(t, z)dydz.
\end{align}
Now, we estimate each term on the right-hand side, individually. The first term on the right hand side of \eqref{Equintp3} can be evaluated, by using $(\Lambda_1)$, \eqref{convexp4} and Lemma \ref{G(T)}, as
\begin{align}\label{est1}
\frac{1}{2} \int_0^R  \int_0^{R-z} & \sigma_2'(h_n(t, y+z)) (y+z)^{-\beta}  A_n^{\zeta}(y, z)g_n(t, y)g_n(t, z) dy dz\nonumber\\
%\le &\frac{1}{2}k_1\int_0^R \int_0^{R-z} y^{-\beta} \frac{(1+y+z)}{(yz)^{\beta}} \sigma_2'(h_n(t, y+z))g_n(t, y)g_n(t, z)dydz\nonumber\\
\le & \frac{1}{2}k_1(1+R) \int_0^R \int_0^{R-z}  y^{-2\beta} z^{-\beta}  \sigma_2'(h_n(t, y+z))g_n(t, y)g_n(t, z) dydz\nonumber\\
\le & \frac{1}{2}k_1(1+R)  \int_0^R \int_0^{R-z}  y^{-2\beta} [\sigma_2(h_n(t, y+z))+\sigma_2 (h_n(t, z))] g_n(t, y)dydz\nonumber\\
%\le &  k_1(1+R) \int_0^R \int_0^{R}  y^{-2\beta} \sigma_2 (h_n(t, z)) g_n(t, y)dydz \nonumber\\
\le &  C_1 (T, R) \int_0^R\sigma_2 (h_n(t, z))dz,
\end{align}
where $C_1(T, R) := k_1(1+R) \mathcal{G}(T)$. Next, the second term can be estimated, by using $(\Lambda_3)$, \eqref{convexp4}, the definition of $\mathcal{C}_{VP, \infty}$ and Lemma \ref{G(T)}, as
\begin{align}\label{est2}
\int_0^R \int_0^z &  y^{-\beta} b(y|z)S_n(z) \sigma_2'(h_n(t, y))g_n(t, z)dydz \nonumber\\
%\le & k_2 (\nu+2) \int_0^R \int_0^z  y^{-\beta +\nu}   \sigma_2'(h_n(t, y))g_n(t, z)dydz \nonumber\\
\le &  k_2 (\nu+2) \int_0^R \int_0^z    [ \sigma_2(h_n(t, y)) +\sigma_2(y^{\nu -\beta})] g_n(t, z)dydz \nonumber\\
%& + k_2 (\nu+2) \int_0^R \int_0^z  \sigma_2(y^{\nu -\beta}) g_n(t, z)dydz \nonumber\\
 \le &  k_2 (\nu+2) \bigg[ \mathcal{G}(T) \int_0^R    \sigma_2(h_n(t, y)) dy  +  \int_0^R \int_0^z  (y^{\nu -\beta} )^{\gamma}  \frac{ \sigma_2(y^{\nu -\beta}) }{(y^{\nu -\beta} )^{\gamma}} g_n(t, z)dydz \bigg] \nonumber\\
 \le &  k_2 (\nu+2) \mathcal{G}(T) \int_0^R    \sigma_2(h_n(t, y)) dy  + k_2 \frac{(\nu+2)}{{ \gamma \nu -\gamma \beta+1} } S_{\gamma} \int_0^R   z^{\gamma \nu -\gamma \beta+1} g_n(t, z)dz \nonumber\\
  \le &  k_2 (\nu+2) \mathcal{G}(T) \int_0^R    \sigma_2(h_n(t, y)) dy  + k_2 \frac{(\nu+2)}{{\gamma \nu -\gamma \beta+1} } S_{\gamma} \mathcal{G}(T).
\end{align}
Finally, we estimate the third term, by using $(\Lambda_2)$, $(\Lambda_3)$, \eqref{convexp4}, the definition of $\mathcal{C}_{VP, \infty}$ and Lemma \ref{G(T)}, as
\begin{align}\label{est3}
\int_R^n \int_0^R &  y^{-\beta} b(y|z)S_n(z) \sigma_2'(h_n(t, y))g_n(t, z)dydz \nonumber\\
\le & k_2 (\nu+2) \int_R^n \int_0^R  y^{-\beta} y^{\nu}   \sigma_2'(h_n(t, y))g_n(t, z)dydz \nonumber\\
%\le & k_2 (\nu+2) \mathcal{G}(T) \int_0^R  y^{\nu-\beta }  \sigma_2'(h_n(t, y)) dy \nonumber\\
%\le & k_2 (\nu+2)  \mathcal{G}(T)   \int_0^R    [\sigma_2(h_n(y, t))+\sigma_2(y^{\nu -\beta})] dy \nonumber\\
\le & k_2 (\nu+2)  \mathcal{G}(T)  \bigg[ \int_0^R   \sigma_2(h_n(t, y))  dy +   \int_0^R     \sigma_2(y^{\nu -\beta}) dy \bigg] \nonumber\\
\le & k_2 (\nu+2)  \mathcal{G}(T)  \bigg[ \int_0^R   \sigma_2(h_n(t, y))  dy + \frac{ S_{\gamma} }{ \gamma \nu -\gamma \beta+1 } R^{\gamma \nu -\gamma \beta+1} \bigg].
\end{align}
Collecting all above estimates in (\ref{est1}), (\ref{est2}) and (\ref{est3}), and inserting them into \eqref{Equintp3}, we have
\begin{align}
\frac{d}{dt}\int_0^R \sigma_2(h_n(t, y))dy\le  C_2(T, R)\int_0^R\sigma_2(h_n(t, z))dz+C_3(T, R),
\end{align}
where $C_2(T, R):= C_1(T, R)+2k_2(\nu +2) \mathcal{G}(T)$ and $C_3(T, R):= k_2 \frac{(\nu +2)}{(\gamma \nu-\gamma \beta+1)} S_\gamma [ \mathcal{G}(T)+R^{\gamma \nu -\gamma \beta+1}] $.
Then applying the Gronwall's inequality, we obtain
\begin{align}
\int_0^R \sigma_2(y^{-\beta}g_n(t, y))dy \le C(T, R),
\end{align}
where $C(T, R)$ is a constant depending on $T$ and $R$. This completes the proof of the Lemma \ref{last} $(i)$. One can infer the second part of Lemma \ref{last} by using \eqref{tmc} and $(\Lambda_4)$. This completes the proof of Lemma \ref{last}.
\end{proof}

%%%%%%%%%%%%%%%%%%%%%%%%%%%%%%%%%%%%%%%%%%%%%%%
%%%%%%%%%%%%%%%%%%%%%%%%%%%%%%%%%%%%%%%%%%%%%%%
\subsection{Equi-continuity w.r.t. time in weak sense}
%%%%%%%%%%%%%%%%%%%%%%%%%%%%%%%%%%%%%%%%%%%%%%
%%%%%%%%%%%%%%%%%%%%%%%%%%%%%%%%%%%%%%%%%%%%%%
\begin{lem}\label{Equicontinuityweaksense}
Assume $(\Lambda_1)$--$(\Lambda_4)$ hold. For any $T>0$ and $R>1$, there is a positive constant $C_5(T, R)$ depending on $T$ and $R$ such that
\begin{align*}
\bigg|\int_0^{\infty}  y^{-\beta} \Psi(y) [g_n(t, y)-g_n(s, y)]dy \bigg|\le C_5(T, R)(t-s),
\end{align*}
for every $n > 1$, $0 \le s \le t \le T$ and $\Psi \in L^{\infty}(0, \infty).$
\end{lem}

\begin{proof}
Let $T>0$ and $R>1$. For $n > 1$, $0 \le s \le t \le T$ and $\Psi \in L^{\infty}(0, \infty)$, we evaluate the following integral as
\begin{align}\label{Equicontinuity1}
\int_0^R y^{-\beta} \Psi(y) & |g_n(t, y)-g_n(s, y)|dy\nonumber\\
 \le & \| \Psi \|_{L^{\infty}(0, \infty)} \int_s^t \int_0^R y^{-\beta} \bigg|  \frac{\partial g_n}{\partial t}(\tau, y) \bigg| dy d\tau \nonumber\\
 \le  & \| \Psi \|_{L^{\infty}(0, \infty)} \int_s^t  \bigg[   \frac{1}{2} \int_0^R  \int_{0}^{y} y^{-\beta} A_n^{\zeta}(y-z, z) g_n(\tau, y-z) g_n(\tau, z)dzdy \nonumber\\
 & + \int_0^R  \int_{0}^{n-y} y^{-\beta} A_n^{\zeta}(y, z) g_n (\tau, y) g_n(\tau, z)dz dy \nonumber\\
  &+\int_0^R  \int_y^{n} y^{-\beta} b(y|z)S_n(z) g_n(\tau, z)dz dy+ \int_0^R  y^{-\beta} S_n(y) g_n (\tau, y) dy \bigg] d\tau.
\end{align}
Next, we evaluate each integral on the right-hand side to \eqref{Equicontinuity1} separately. First, we evaluate the first integral, by using Fubini's theorem, ($\Gamma_1$) and Lemma \ref{G(T)} as
\begin{align}\label{Equicontinuity2}
 \frac{1}{2} \int_s^t \int_0^R  \int_{0}^{y} & y^{-\beta} A_n^{\zeta}(y-z, z) g_n(\tau, y-z) g_n(\tau, z)dzdy d\tau \nonumber\\
%\le &   \frac{1}{2} \int_s^t \int_0^R  \int_{0}^{R-z} (y+z)^{-\beta} A_n^{\zeta}(y, z) g_n(\tau, y) g_n(\tau, z) dy dz d\tau \nonumber\\
\le  &  \frac{1}{2}k_1 \int_s^t \int_0^R  \int_{0}^{R-z} (y+z)^{-\beta} \frac{(1+y+z)}{(yz)^{\beta}} g_n(\tau, y) g_n(\tau, z) dy dz d\tau \nonumber\\
\le  &  \frac{1}{2}k_1 (1+R)\int_s^t \int_0^R  \int_{0}^R y^{-2\beta} z^{-\beta} g_n(\tau, y) g_n(\tau, z) dy dz d\tau \nonumber\\
\le &   \frac{1}{2}k_1 (1+R) \mathcal{G}^2(T) (t-s).
\end{align}
Similarly, by applying ($\Gamma_1$) and Lemma \ref{G(T)}, the second integral can be estimated, as
\begin{align}\label{Equicontinuity3}
 \int_s^t \int_0^R & \int_{0}^{n-y}  y^{-\beta} A^{\zeta}_n(y, z) g_n(\tau, y) g_n(\tau, z) dz dy d\tau \nonumber\\
\le & k_1  \int_s^t \int_0^R  \int_{0}^{n}  y^{-\beta}  \frac{(1+R+z)}{(yz)^{\beta}}  g_n(\tau, y) g_n(\tau, z) dz dy d\tau \nonumber\\
\le & k_1 \mathcal{G}(T)  \int_s^t  \int_{0}^{n}  (1+R+z) z^{-\beta}  g_n(\tau, z)dz d\tau
\le  2 k_1 (1+R) \mathcal{G}^2(T)(t-s).
\end{align}
We evaluate the third integral, by using Fubibi's theorem, ($\Gamma_2$), ($\Gamma_3$), and Lemma \ref{G(T)}, as
\begin{align}\label{Equicontinuity4}
 \int_s^t \int_0^R  \int_y^{n} & y^{-\beta} b(y|z)S_n(z) g_n(\tau, z)dz dy d\tau \nonumber\\
 \le & \int_s^t \int_0^n  \int_0^{z} y^{-\beta} b(y|z) S_n(z) g_n(\tau, z)dy dz d\tau \nonumber\\
  \le & c_1 k_2 \int_s^t \int_0^n  z^{1+\nu-\beta}  g_n(\tau, z) dz d\tau \le c_1 k_2 \mathcal{G}(T)(t-s).
\end{align}
Finally, the last term can be estimated, by applying ($\Gamma_3$), and Lemma \ref{G(T)}, as
\begin{align}\label{Equicontinuity5}
 \int_s^t \int_0^R  & y^{-\beta} S_n(y) g_n (\tau, y) dy d\tau
  \le  k_2 \int_s^t \int_0^R  y^{1+\nu-\beta}  g_n(\tau, y) dy d\tau \le  k_2 \mathcal{G}(T)(t-s).
\end{align}
Inserting \eqref{Equicontinuity2}, \eqref{Equicontinuity3}, \eqref{Equicontinuity4} and \eqref{Equicontinuity5} into \eqref{Equicontinuity1}, we have
\begin{align}\label{Equicontinuity6}
\int_0^R y^{-\beta} & \Psi(y)  |g_n(t, y)-g_n(s, y)|dy \le  C_5(T, R) (t-s),
\end{align}
where
\begin{align*}
C_5(T, R)=\| \Psi \|_{L^{\infty}(0, \infty)} \bigg[ \frac{1}{2}k_1 (1+R) \mathcal{G}(T) + 2 k_1 (1+R) \mathcal{G}(T)+ c_1 k_2 + k_2 \bigg]\mathcal{G}(T).
\end{align*}
Now for arbitrary $\epsilon >0$, we evaluate the following integral, by using \eqref{Equicontinuity6} and Lemma \ref{last}, as
\begin{align}\label{Equicontinuity7}
\bigg|\int_0^{\infty} & y^{-\beta} \Psi(y) [g_n(t, y)-g_n(s, y)]dy \bigg| \nonumber \\
\le & \bigg|\int_0^{R}  y^{-\beta} \Psi(y) [g_n(t, y)-g_n(s, y)]dy \bigg| + \bigg|\int_R^{\infty}  y^{-\beta} \Psi(y) [g_n(t, y)-g_n(s, y)]dy \bigg| \nonumber\\
\le & C_5(T, R) (t-s)+2 \| \Psi \|_{L^{\infty}(0, \infty)} \epsilon.
\end{align}

This completes the proof of Lemma \ref{Equicontinuityweaksense}.
\end{proof}

We are now in a position to complete the proof of Theorem \ref{TheoremCmfe} in the next subsection.

%%%%%%%%%%%%%%%%%%%%%%%%%%%%%%%%%%%%%%%%%%%%%%%
%%%%%%%%%%%%%%%%%%%%%%%%%%%%%%%%%%%%%%%%%%%%%%%
\subsection{Convergence of integrals}
%%%%%%%%%%%%%%%%%%%%%%%%%%%%%%%%%%%%%%%%%%%%%%
%%%%%%%%%%%%%%%%%%%%%%%%%%%%%%%%%%%%%%%%%%%%%%
\emph{Proof of Theorem} \ref{TheoremCmfe}:\
From de la Vall\`{e}e Poussin theorem, Lemma (\ref{G(T)})--(\ref{last}), and then using Dunford-Pettis theorem and a variant of the Arzel\`{a}-Ascoli theorem, see \cite{Vrabie:1995}, we conclude that $(g_n)$ is relatively compact in $\mathcal{C}([0, T]_w; L_{-\beta}^1(\mathds{R}_{>0}))$ for each $T>0$. There is thus a subsequence of $(g_n)$ and a nonnegative function $g\in \mathcal{C}([0, T]_w; L_{-\beta}^1(\mathds{R}_{>0}))$ such that
\begin{align}\label{weakconvergence1}
 g_n \to g \ \ \ \text{in} \ \ \mathcal{C}([0, T]_w: L^1(\mathds{R}_{>0}); \zeta^{-\beta} d \zeta)
\end{align}
for each $T>0$. \\

Next, we can improve the weak convergence \eqref{weakconvergence1}, by applying Lemma (\ref{G(T)}), (\ref{C(T)}) and \eqref{weakconvergence1}, as
\begin{align}\label{weakconvergence2}
 g_n \to g\ \ {in} \ \ \mathcal{C}([0,T]_w: L^1(\mathds{R}_{>0}); (\zeta^{-\beta} +\zeta) d \zeta).
\end{align}

In order to show that $g$ is actually a solution to \eqref{Cmfe}--\eqref{Initialdata} in the sense of \eqref{definition}, it remains to verify all the truncated integrals in \eqref{tcfe} convergence weakly to the original integrals in \eqref{Cmfe}. This is now a standard procedure to prove this convergence of integrals, see \cite{Camejo:2015, Barik:2017Anote, Barik:2018Mass, Giri:2012, Laurencot:2002, Laurencot:2015, Laurencot:2018, Stewart:1989}. Thus, $g$ is a weak solution to \eqref{Cmfe}--\eqref{Initialdata}.\\

 Finally to complete the proof of Theorem \ref{TheoremCmfe}, it is required to show that $g$ is a mass-conserving solution to \eqref{Cmfe}--\eqref{Initialdata}. For the case of non-conservative one ($\theta=0$), one can be  easily proved as similar to \cite{Filbet:2004, Barik:2017Anote} and on the other hand, for conservative case ($\theta=0$), we infer from \eqref{weakconvergence2} and \eqref{tmc}, which completes the proof of Theorem \ref{TheoremCmfe}.

\section*{Acknowledgments}
The author would like to thank Professor Swadhin Pattanayak, former director of Institute of Mathematics and Applications, Bhubaneswer, for his valuable suggestions and comments.
%%%%%%%%%%%%%%%%%%%%%%%%%%%%%%%%%%%%%%%%%%%%%%%%%%%%%%%%%%%%%%%%%%%
%%%%%%%%%%%%%%%%%%%%%%%%%%%%%%%%%%%%%%%%%%%%%%%%%%%%%%%%%%%%%%%%%%%
% The author would like to thank University Grant Commission (UGC), India, for providing Ph.D fellowship.
%%%%%%%%%%%%%%%%%%%%%%%%%%%%%%%%%%%%%%%%%%%%%%%%%%%%%%%%%%%%%%%%%%%
%%%%%%%%%%%%%%%%%%%%%%%%%%%%%%%%%%%%%%%%%%%%%%%%%%%%%%%%%%%%%%%%%%%
\bibliographystyle{plain}

%%%%%%%%%%%%%%%%%%%%%%%%%%%%%%%%%
%%%%%%%%%%%%%%%%%%%%%%%%%%%%%%%%%%%%%%%%%%%%%%%%%%%%%%%%%%%%%%%%%%%
\end{document}